\documentclass[11pt, leqno]{amsart}
\usepackage{graphicx}
\usepackage{amsfonts,delarray,amssymb,amsmath,amsthm,a4,a4wide}
\usepackage{latexsym}
\usepackage{epsfig}
\usepackage{color}
\usepackage[margin=1in]{geometry} 
\newtheorem{thm}{Theorem}[section]
\newtheorem{cor}[thm]{Corollary}
\newtheorem{lem}[thm]{Lemma}

\theoremstyle{definition}

\newtheorem{rem}[thm]{Remark}
\numberwithin{equation}{section}

\newcommand{\dist}{\mbox{dist}\,}

\newcommand{\R}{\mathbb R}

\newcommand{\e}{\varepsilon}

\newcommand{\ov}{\overline}
\newcommand{\p}{\partial}

\newcommand{\comment}[1]{}

\newenvironment{myindentpar}[1]%
{\begin{list}{}%
         {\setlength{\leftmargin}{#1}}%
         \item[]%
}
{\end{list}}

\begin{document}

\title[Convex functionals without uniform convexity]{On approximating minimizers of convex functionals with a convexity constraint by singular Abreu equations without uniform convexity}
\author{Nam Q. Le}
\address{Department of Mathematics, Indiana University,
Bloomington, 831 E 3rd St, IN 47405, USA.}
\email{nqle@indiana.edu}

\thanks{The research of the author was supported in part by the National Science Foundation under grant DMS-1764248}

\begin{abstract}
We revisit the problem of approximating minimizers of certain convex functionals subject to a convexity constraint by  solutions of fourth order equations of Abreu type.
 This approximation problem was studied in previous works of Carlier-Radice (Approximation of variational problems with a convexity constraint by PDEs of Abreu type.
{\it Calc. Var. Partial Differential Equations.} {\bf 58} (2019), no. 5, Art. 170) and the author  (Singular Abreu equations and minimizers of convex functionals with a convexity constraint, arXiv:1811.02355v3, {\it  Comm. Pure Appl. Math.}, to appear), under the uniform convexity of both the Lagrangian and constraint barrier. By introducing a new approximating scheme, we completely remove the uniform convexity of both the Lagrangian and 
constraint barrier.  Our analysis is applicable to variational problems motivated by the original 2D Rochet-Chon\'e model in the monopolist's problem in Economics, and variational problems arising in 
the analysis of wrinkling patterns in floating elastic shells in Elasticity.
\end{abstract}
\maketitle

\section{Introduction}

In this note, we revisit the problem of approximating minimizers of certain convex functionals subject to a convexity constraint by  solutions of fourth order equations of Abreu type. This problem was investigated in previous works by Carlier-Radice \cite{CR} and the author \cite{Le18}, under the uniform convexity of both the Lagrangian and constraint barrier. Here, 
by introducing a new approximating scheme, we completely remove the uniform convexity of both the Lagrangian and 
constraint barrier.
We start by recalling this problem.
\subsection{Approximating minimizers of convex functionals subject to a convexity constraint}
\label{Q_sect}
Let $\Omega_0$ be a bounded, open, smooth, and convex domain in $\R^n$ ($n\geq 2$). Let $\Omega$ be a bounded, open, smooth, uniformly convex domain containing $\overline{\Omega_0}$.
Let $\varphi$ be a convex and smooth function defined in $\overline{\Omega}$.  Let $F(x, z, p):\R^n\times \R\times \R^n\rightarrow\R$ be a smooth Lagrangian which is convex in each of the variables $z\in\R$ and $p=(p_1,\cdots, p_n)\in\R^n$.
Consider  the following variational problem with a convexity constraint:
\begin{equation}
\label{p1}
\inf_{u\in \bar{S}[\varphi,\Omega_0]} \int_{\Omega_0} F(x, u(x), Du(x)) dx
\end{equation}
where
\begin{multline}
\label{barS}
\bar{S}[\varphi, \Omega_0]=\{ u: \Omega_0\rightarrow \R\mid u \text{ is convex, }\\ \text{ u admits a convex extension to }\Omega \text{ such that } u=\varphi\text{ on }\Omega\setminus \Omega_0\}.
\end{multline}
Note that elements of $\bar{S}[\varphi, \Omega_0]$ are Lipschitz continuous with Lipschitz constants bound from above by $\|D\varphi\|_{L^{\infty}(\Omega)}$ and hence $\bar{S}[\varphi, \Omega_0]$ is compact in the topology of uniform convergence. Under quite general assumptions on the convexity and growth of the Lagrangian $F$, one can show that (\ref{p1}) has a minimizer in $\bar{S}[\varphi, \Omega_0]$.

Due to the intrinsic difficulty of the convexity constraint, as elucidated in \cite{CR, Le18}, for practical purposes such as numerical computations, one wonders if 
minimizers of (\ref{p1}) can be well approximated in the uniform norm by solutions of some higher order equations whose global well-posedness can be established. The approximating schemes proposed in \cite{CR, Le18} use the second boundary value problem of fourth order equations of Abreu type which we now would like to make more precise.

Let $\psi$ be a smooth function in $\overline{\Omega}$ with $\inf_{\p\Omega}\psi>0$.
Fix $0\leq\theta<1/n$.
 For each $\e>0$,
 consider the following second boundary value problem for a uniform convex function $u_\e$:
\begin{equation}
  \left\{ 
  \begin{alignedat}{2}\e\sum_{i, j=1}^{n}U_\e^{ij}(w_\e)_{ij}~& =f_{\e}(\cdot, u_\e, Du_\e, D^2 u_\e;\varphi)~&&\text{in} ~\Omega, \\\
 w_\e~&= (\det D^2 u_\e)^{\theta-1}~&&\text{in}~ \Omega,\\\
u_\e ~&=\varphi~&&\text{on}~\p \Omega,\\\
w_\e ~&= \psi~&&\text{on}~\p \Omega.
\end{alignedat}
\right.
\label{Abe0}
\end{equation}
Here and what follows, $U_\e=(U_\e^{ij})_{1\leq i, j\leq n}$ is the cofactor matrix of the Hessian matrix $$D^2 u_\e=\left((u_\e)_{ij}\right)_{1\leq i, j\leq n}\equiv \left(\frac{\p^2 u_\e}{\p x_i \p x_j}\right)_{1\leq i, j\leq n}$$ and
\begin{multline}
\label{fe0}
f_{\e}(x, u_\e(x), Du_\e(x), D^2 u_\e(x);\varphi(x)) \\= \left\{\begin{array}{rl}
  \frac{\p F}{\p z}(x, u_\e(x), Du_\e(x)) -\displaystyle \sum_{i=1}^n \frac{\p}{\p x_i} \left(\frac{\p F}{\p p_i}(x, u_\e(x), Du_\e(x))\right)&  x\in \Omega_0,\\
\frac{1}{\e}(u_\e (x)-\varphi(x) ) & x\in  \Omega\setminus \Omega_0.
\end{array}\right.
\end{multline}
The fourth order expression $U^{ij} \left[(\det D^2 u)^{\theta-1}\right]_{ij}$ appears in several geometric contexts including K\"ahler geometry (such as the Abreu's equation when $\theta=0$; see \cite{Ab}) and affine geometry (such as the affine maximal surface equation when 
$\theta=\frac{1}{n+2}$; see \cite{TW00}). 
When the Lagrangian $F$ depends on the  gradient variables, the right hand side of (\ref{fe0}) contains the Hessian $D^2 u_\e$ of $u_\e$. Without further regularity for the convex function $u_\e$, the Hessian $D^2 u_\e$ can be just a measure-valued matrix. Thus, as in \cite{Le18}, we call fourth order equations of the type (\ref{Abe0})-(\ref{fe0}) singular Abreu equations. 
 
We note that the first two equations of the system (\ref{Abe0})-(\ref{fe0}) are critical points, with respect to compactly supported variations,  of  the following functional 
\begin{equation*}
J_{\e}(v)=\int_{\Omega_0}  F(x, v(x), Dv(x))dx +\frac{1}{2\e}\int_{\Omega\setminus\Omega_0} (v-\varphi)^2 dx-\e\int_{\Omega}  \frac{(\det D^2 v)^\theta-1}{\theta} dx.
\end{equation*}
When $\theta=0$, the integral $\int_{\Omega}  \frac{(\det D^2 v)^\theta-1}{\theta} dx$ is replaced by $\int_{\Omega}\log\det D^2 v dx.$ The requirement $0\leq \theta<1/n$ is to make $J_\e$ a convex functional.

The function $f_\e$ defined by (\ref{fe0}) is not continuous in general; this is usually due to the jump discontinuity through $\p\Omega_0$. Thus, the best global regularity one can expect for a solution to  (\ref{Abe0})-(\ref{fe0}) is $W^{4,p}(\Omega)$ for all $p<\infty$.

The questions we would like to ask are the following:
\begin{myindentpar}{1cm}
\it
({\bf Q1}) Does the system (\ref{Abe0})-(\ref{fe0}) have a uniformly convex solution $u_\e\in W^{4,p}(\Omega)$ (for all $p<\infty$) for each $\e>0$ small? \\
({\bf Q2}) If the answer to {\bf Q1} is yes, 
does $u_\e$ converge uniformly on compact subsets of $\Omega$ to a minimizer $u\in \bar{S}[\varphi,\Omega_0]$ of problem (\ref{p1})?
\end{myindentpar}

 Another way to rephrase the above questions is to study limiting properties of solutions, if any, to singular Abreu equations of the type (\ref{Abe0})-(\ref{fe0})  when $\e\rightarrow 0$.

The positive answers to questions {\bf Q1} and {\bf Q2} above have been given in  \cite[Theorem 5.3]{CR} and \cite[Theorem 2.3]{Le18} when $F$ and $\varphi$ satisfy certain structural conditions. These work require the uniform convexity of the Lagrangian $F(x, z, p)$ with respect to $z$
and also the uniform convexity of the barrier constraint $\varphi$. We recall these theorems here.
\begin{thm}(\cite[Theorem 5.3]{CR}) 
\label{CRthm} 
Let $\theta=0$. Let $\psi$ be a smooth function in $\overline{\Omega}$ with $\inf_{\p\Omega}\psi>0$.
Assume that $\varphi$ is uniformly convex in $\overline{\Omega}$ and that $F(x, z, p)= F^0(x, z)$ where $F^0$ is uniformly convex with respect to $z$, that is, 
$f^0(x, z):=\frac{\p F^0(x, z)}{\p z}$ satisfies for some $\alpha>0$
\begin{equation}
\label{f01}
(f^0(x, z)- f^0(x, \tilde z))(z-\tilde z)\geq \alpha |z-\tilde z|^2 ~\text{for all } x\in\Omega_0~\text{and all } z,\tilde z\in\R.
\end{equation}
Assume that, for  some continuous and increasing function $\eta:[0,\infty)\rightarrow [0,\infty)$, we have
\begin{equation}
\label{f02}
|f^0(x, z)|\leq \eta (|z|)~\text{for all } x\in\Omega_0~\text{and all } z\in\R.
\end{equation}
Then, for $\e>0$ small, the system (\ref{Abe0})-(\ref{fe0}) has  a uniformly convex solution $u_\e \in W^{4,q}(\Omega)$  for all $q\in (n,\infty)$. Moreover, when $\e\rightarrow 0$, $u_\e$ converges uniformly on compact subsets of $\Omega$ to the unique minimizer $u\in \bar{S}[\varphi,\Omega_0]$ of (\ref{p1}). 
\end{thm}

\begin{thm}(\cite[Theorem 2.3]{Le18}) 
\label{Lthm}
Assume $n=2$ and $0\leq \theta<1/n$. Let $\psi$ be a smooth function in $\overline{\Omega}$ with $\inf_{\p\Omega}\psi>0$.
Assume that $\varphi$ is uniformly convex in $\overline{\Omega}$ and that $F(x, z, p)= F^0(x, z) + F^1(x, p)$ where $F^0$  satisfies (\ref{f01}) and (\ref{f02}). 
Suppose that for some $M\geq 0$, we have for all $p\in\R^n$,
$$0\leq F^1_{p_i p_j}(x, p)\leq M I_n; |F^1_{p_i x_i}(x, p)| \leq  M (|p| +1) \text{ for all } x\in\Omega_0~\text{and for each } i.$$
Then, for $\e>0$ small and $\alpha>0$ sufficiently large, the system (\ref{Abe0})-(\ref{fe0}) has  a uniformly convex solution $u_\e \in W^{4,q}(\Omega)$  for all $q\in (n,\infty)$. Moreover, for $\alpha$ sufficiently large,  $u_\e$ converges, when $\e\rightarrow 0$, uniformly on compact subsets of $\Omega$ to the unique minimizer $u\in \bar{S}[\varphi,\Omega_0]$ of  (\ref{p1}). 
\end{thm}
In Theorem \ref{Lthm} and what follows, we use the following notation: $I_n$ is the identity $n\times n$ matrix and
$$F^{1}_{p_i p_j}(x, p)=\frac{\p^2 F^1(x, p)}{\p p_i \p p_j};\quad
F^{1}_{p_i x_j}(x, p)=\frac{\p^2 F^1(x, p)}{\p p_i \p x_j}; \nabla_p F^1(x, p)=\left(\frac{\p F^1(x, p)}{\p p_1},\cdots, \frac{\p F^1(x, p)}{\p p_n}\right).$$
\begin{rem}
Inspecting the proof of Theorem 5.3 in \cite{CR}, we find that Theorem \ref{CRthm} also holds for all $\theta \in [0, 1/n)$.
\end{rem}
From the variational analysis and practical models in Economics and Elasticity to be described below, it would be interesting to remove the uniform convexity assumptions in Theorems \ref{CRthm} and \ref{Lthm}. 
\begin{center}
\it
({\bf Q3}) Can we remove the uniform convexity assumptions on $F$ and $\varphi$  in Theorems \ref{CRthm} and \ref{Lthm}?
\end{center}

\subsection{Examples with non-uniformly convex Lagrangians} Our examples of convex functionals subject to a convexity constraint  arise in the Rochet-Chon\'e model of the 
monopolist's problem in Economics and variational problems arising in 
the analysis of wrinkling patterns in floating elastic shells in Elasticity. In these models, the Lagrangians $F(x, z, p)$ are convex but not uniformly convex with respect to $z$.
\\
{\it  The Rochet-Chon\'e model.}
The analysis in \cite{Le18} is applicable to the 2D 
 Rochet-Chon\'e model  perturbed by a strictly convex lower order term. It is not known if the analysis in \cite{Le18} is applicable to the original Rochet-Chon\'e model \cite{RC} where
 $$F(x, z, p)=\frac{1}{2}|p|^2 \gamma(x) -x\cdot p \gamma(x) + z \gamma (x).$$
  Rochet-Chon\'e modeled the monopolist problem in product line design with quadratic cost using maximization of the functional
$$\Phi(u)= \int_{\Omega_0} \{x\cdot D u(x) - \frac{1}{2}|D u(x)|^2- u(x)\} \gamma(x) dx.$$
Here $ \Phi(u)$ is  the monopolist's profit; 
u is the buyers' indirect utility function with bilinear valuation;
  $\Omega_0\subset\R^n$ is the collection of  types of agents; $\gamma$ is the relative frequency of different types of agents in the population. The function $\gamma$ is assumed to be nonnegative, bounded and Lipschitz continuous, that is,
  $$0\leq \gamma\leq C\quad\text{and } \|D \gamma\|_{L^{\infty}(\Omega_0)}\leq C.$$
  For a consumer of type $x\in\Omega_0$, the indirect utility $u(x)$ is computed via the formula 
$$u(x) =\max_{q\in Q} \{ x\cdot q-p(q)\}$$
where $Q\subset\R^n$ is the product line and $p: Q\rightarrow \R$ is a price schedule that the monopolist needs to both design to maximize her total profit $\Phi$. 
 Clearly, $u$  is convex and maximizing $\Phi(u)$ is equivalent to minimizing $\int_{\Omega_0} F(x, u(x), Du(x))$ among all convex functions $u$. For economic reasons, there are other conditions for $u$ outside $\Omega_0$; see \cite{RC} and also \cite{CLR1} for more details.  \\
 {\it Thin elastic shells.}
We also note that, in certain applications where $F$ is independent of the gradient variables, $F$ can be non-uniformly convex in $z$. A particular example arises in the analysis of wrinkling patterns in floating elastic shells by Tobasco \cite{T}. As discussed in \cite[Section 1.2.3]{T}, describing the leading order behavior of weakly curved floating shells lead to limiting problems which are dual to problems of the type:

Given a smooth function $q:\overline{\Omega_0}\subset\R^2\rightarrow \R$, minimize
\begin{equation}
\label{elas}
\int_{\Omega_0} \left(\frac{|x|^2}{2}-u(x)\right)\det D^2 q(x)~ dx
\end{equation}
over the set
$$\{u~\text{convex in }\R^2,~u=\frac{|x|^2}{2}~\text{in }\R^2\setminus \Omega_0\}.$$
Optimal functions in (\ref{elas}) are called optimal {\it Airy potential} in \cite{T}.
In this example,
$$F(x, z, p)=  \left(\frac{|x|^2}{2}-z\right)\det D^2 q(x).$$

\subsection{The main results}
In this note, we answer question {\bf Q3} at the end of section \ref{Q_sect} by completely removing both the uniform convexity of $F$ with respect to $z$ and the uniform convexity of $\varphi$. To do this, we introduce a new approximating scheme, slightly different from (\ref{Abe0})-(\ref{fe0}).

As in \cite{Le18} and motivated by the Rochet-Chon\'e model, we consider 
 Lagrangians of the form:
$$F(x, z, p)= F^0(x, z) + F^1 (x, p).$$
Let $$f^0(x, z):=\frac{\p F^0(x, z)}{\p z}.$$
We assume the following convexity and growth assumptions on $F^0$ and $F^1$. For some {\it nonnegative} constant $C_\ast$:
\begin{equation}
\label{F0}
(f^0(x, z)-f^0(x, \tilde z))(z-\tilde z)\geq 0; |f^0(x, z)|\leq \eta (|z|)~\text{for all } x\in\Omega_0~\text{and all } z,\tilde z\in\R
\end{equation}
where $\eta:[0,\infty)\rightarrow [0,\infty)$ is a continuous and increasing function.  Furthermore, for all $p\in\R^n$
\begin{equation}
\label{F1}
0\leq F^1_{p_i p_j}(x, p)\leq C_\ast I_n; |F^1_{p_i x_i}(x, p)| \leq  C_\ast (|p| +1) \text{ for all }x\in\Omega_0~\text{and for each } i.
\end{equation}
Let $\rho$ be a strictly convex defining function of $\Omega$, that is, 
\begin{equation}
\label{rhoeq}
\Omega:=\{x\in \R^n: \rho(x)<0\},~\rho=0 \text{ on } \p\Omega \text{ and }D\rho\neq 0 \text{ on }\p\Omega.
\end{equation} Let
\begin{equation}
\label{Cvar}
C_{\varphi} \\= \left\{\begin{array}{rl}
  0&  \text{if }\varphi \text{ is uniformly convex in } \overline{\Omega},\\
1 & \text{otherwise}.
\end{array}\right.
\end{equation}
 For $\e>0$,
 consider the following second boundary value problem for a uniform convex function $u_\e$:
\begin{equation}
  \left\{ 
  \begin{alignedat}{2}\e\sum_{i, j=1}^{n}U_\e^{ij}(w_\e)_{ij}~& =f_{\e}\left(\cdot, u_\e, Du_\e, D^2 u_\e; \varphi+ C_{\varphi} \e^{\frac{1}{3n^2}} (e^{\rho}-1)\right)~&&\text{in} ~\Omega, \\\
 w_\e~&= (\det D^2 u_\e)^{\theta-1}~&&\text{in}~ \Omega,\\\
u_\e ~&=\varphi~&&\text{on}~\p \Omega,\\\
w_\e ~&= \psi~&&\text{on}~\p \Omega.
\end{alignedat}
\right.
\label{Abe}
\end{equation}
Here,
\begin{multline}
\label{fe}
f_{\e}(x, u_\e(x), Du_\e(x), D^2 u_\e(x); \varphi (x)+ C_{\varphi} \e^{\frac{1}{3n^2}} (e^{\rho(x)}-1)) \\= \left\{\begin{array}{rl}
  \frac{\p F}{\p z}(x, u_\e(x), Du_\e(x)) -\displaystyle \sum_{i=1}^n \frac{\p}{\p x_i} \left(\frac{\p F}{\p p_i}(x, u_\e(x), Du_\e(x))\right)&  x\in \Omega_0,\\
\frac{1}{\e}\left(u_\e (x)-\varphi(x)-C_{\varphi} \e^{\frac{1}{3n^2}} (e^{\rho(x)}-1) \right) & x\in  \Omega\setminus \Omega_0.
\end{array}\right.
\end{multline}

Our main theorem states as follows.
\begin{thm}
\label{mthm}
Let $\Omega_0$ and $\Omega$ be bounded, open, smooth, and convex domains in $\R^n$ ($n\geq 2$) such that $\Omega$ is uniformly convex and contains $\overline{\Omega_0}$. Fix $0\leq\theta<1/n$.
Let $\psi$ be a smooth function in $\overline{\Omega}$ with $\inf_{\p\Omega}\psi>0$.
Let $\varphi$ be a convex and smooth function defined in $\overline{\Omega}$. 
Assume that (\ref{F0}) and (\ref{F1}) are satisfied. If $F^1\not\equiv 0$ then we assume further that $n=2$. 
Then the following hold.
\begin{myindentpar}{1cm}
(i) For $\e>0$ small, the system (\ref{Abe})-(\ref{fe}) has  a uniformly convex solution $u_\e \in W^{4,q}(\Omega)$  for all $q\in (n,\infty)$.\\
(ii) For $\e>0$ small, let $u_\e\in W^{4, q}(\Omega)$ $(q>n)$ be a solution to (\ref{Abe})-(\ref{fe}). After extracting a subsequence, $u_\e$ converges uniformly on compact subsets of $\Omega$ to a minimizer $u\in \bar{S}[\varphi,\Omega_0]$  of (\ref{p1}).  
\end{myindentpar}
\end{thm}
Several remarks are in order.
\begin{rem} 
 Without the uniform convexity of $F$ with respect to $z$, minimizers of the problem (\ref{p1}) with the convexity constraint (\ref{barS}) can be non-unique.
As such, each convergent subsequence of $\{u_\e\}$ converges to a minimizer of (\ref{p1}) as stated in Theorem \ref{mthm} (ii). It would be interesting to investigate whether the approximating scheme (\ref{Abe})-(\ref{fe})
selects a distinguished minimizer of the problem (\ref{p1}) when it has several minimizers.
\end{rem}
\begin{rem}
When $\varphi$ is not uniformly convex, the addition of $ \e^{\frac{1}{3n^2}} (e^{\rho}-1)$ to $\varphi$ is to make the new function  ``sufficiently'' uniformly convex. The choice of the exponent $\frac{1}{3n^2}$ (or any positive number not larger than this) is motivated by the need to establish uniform bounds for $u_\e$ in the a priori estimates for solutions to  (\ref{Abe})-(\ref{fe}); see (\ref{Cd}) and (\ref{fe_u2}).
\end{rem}
\begin{rem} Let $G(t)$ be an antiderivative of $t^{\theta-1}$. 
One of the crucial information in the proof of the convergence of solutions of (\ref{Abe0})-(\ref{fe0}) to a minimizer of (\ref{p1}) is a variant of the estimate
 \begin{equation}
\label{G0}
\liminf_{\e} \e\int_{\Omega} G(\det D^2 v)dx\geq 0 ~\text{for all } v\in \bar S[\varphi,\Omega_0].
\end{equation}
\begin{myindentpar}{1cm}
(a) When $\varphi$
 is uniformly convex in $\overline{\Omega}$, for any $v\in \bar S[\varphi,\Omega_0]$, (\ref{G0}) was shown to be true with $v$ being replaced by $(1-\e)v + \e\varphi \in \bar S[\varphi,\Omega_0]$
 in \cite[Proposition 3.5]{CR} and \cite[inequality (5.15)]{Le18}.
 When the uniform convexity of $\varphi$ is removed,  unless $\theta>0$, (\ref{G0}) might fail for all $v\in \bar S[\varphi,\Omega_0]$ as in the case $\varphi$ being a constant for which $\bar S[\varphi,\Omega_0]=\{\varphi\}$.\\
(b) On the other hand,  for any convex $\varphi$, the estimate (\ref{G0}) holds for $v\in \bar S[\varphi,\Omega_0]$ being replaced by $v + C_{\varphi} \e^{\frac{1}{3n^2}} (e^{\rho}-1)$; see (\ref{Gpos}). This somehow indicates the advantage of our approximating scheme.  
\end{myindentpar}
\end{rem}

In Theorem \ref{mthm} and in two dimensions, we can replace the convexity of $F^0$ in (\ref{F0}) by a semi-convexity condition as long as the function $F^1$ is highly uniformly convex with respect to $p$. Moreover, the whole sequence of solutions $u_\e$ to (\ref{Abe})-(\ref{fe}) converges to the unique minimizer $u\in \bar{S}[\varphi,\Omega_0]$  of (\ref{p1}).
This is the content of the next theorem.

\begin{thm}
\label{ethm}
Let $n=2$.
Let $\Omega_0$ and $\Omega$ be bounded, open, smooth, and convex domains in $\R^n$ such that $\Omega$ is uniformly convex and contains $\overline{\Omega_0}$. Fix $0\leq\theta<1/n$.
Let $\psi$ be a smooth function in $\overline{\Omega}$ with $\inf_{\p\Omega}\psi>0$.
Let $\varphi$ be a convex and smooth function defined in $\overline{\Omega}$. 
Assume that the following conditions (\ref{F0l}) and (\ref{F1s}) are satisfied for some positive constants $C_b, C_l, \underbar C, C_\ast$:
\begin{equation}
\label{F0l}
\frac{\p^2 F^0}{\p z^2} (x, z)= \frac{\p f^0}{\p z}(x, z)\geq -C_b,~|f^0(x, z)| \leq C_l (1+ |z|)~~\text{for all } x\in\Omega_0~\text{and all } z\in\R;
\end{equation}
\begin{equation}
\label{F1s}
\underbar C I_2 \leq F^1_{p_i p_j}(x, p)\leq C_\ast I_2; |F^1_{p_i x_i}(x, p)| \leq  C_\ast (|p| +1) \forall x\in\Omega_0, \forall p\in\R^n~\text{and for each } i.
\end{equation}
Then the following hold.
\begin{myindentpar}{1cm}
(i) For $\e>0$ small, the system (\ref{Abe})-(\ref{fe}) has  a uniformly convex solution $u_\e \in W^{4,q}(\Omega)$  for all $q\in (n,\infty)$.\\
(ii) For $\e>0$ small, let $u_\e\in W^{4, q}(\Omega)$ $(q>n)$ be a solution to (\ref{Abe})-(\ref{fe}). 
Assume that $\underbar C$ is large (depending only on $C_b$ and $\Omega_0$).  When $\e\rightarrow 0$, the sequence
$\{u_\e\}$ converges uniformly on compact subsets of $\Omega$ to the unique minimizer $u\in \bar{S}[\varphi,\Omega_0]$  of (\ref{p1}).  
\end{myindentpar}
\end{thm}
\begin{rem}
As explained in detail in \cite[Section 1.3]{Le18}, for a gradient-dependent 
Lagrangian $F$, nothing is known in dimensions $n\geq 3$ about the solvability of the singular Abreu equations (\ref{Abe})-(\ref{fe}) in suitable Sobolev spaces. This is the main reason why we restrict ourselves in this paper to dimensions $n=2$ when the Lagrangians $F$ depend on the gradient variables $p$. 

\end{rem}

Key in the proof of the existence of a uniformly convex solution $u_\e \in W^{4,q}(\Omega)$ to the system (\ref{Abe})-(\ref{fe}) is the derivation of a priori estimates. Crucial ingredients in the convergence proof of $u_\e$ are their uniform a priori estimates with respect to $\e$ small. The uniform convexity of $F$ with respect to $z$ in \cite{CR, Le18} allows us to control $\|u_\e\|_{L^{\infty}(\Omega_0)}$. Here, without the uniform convexity of $F$ with respect to $z$, our new input is that we can control 
 $\|u_\e\|_{L^{\infty}(\Omega_0)}$ by $\|\varphi\|_{L^{\infty}(\Omega)} + \frac{1}{\e} \int_{\Omega\setminus \Omega_0}|u_\e-\varphi|^2 dx$. This follows from Lemma \ref{clem} which is of independent interest.
 \vglue 0.2cm
 
 The rest of the note is organized as follows. In Section \ref{cvx_sec}, we prove a simple but crucial convexity result stated in Lemma \ref{clem}. In Section \ref{pf_sec}, we prove our main results stated in Theorems \ref{mthm} and \ref{ethm}.
\section{A convexity lemma}
\label{cvx_sec}
\begin{lem}
\label{clem}
Let $\Omega_0\subset\subset \Omega_1\subset\subset \Omega_2$ be bounded, convex domains in $\R^n$ ($n\geq 2$). Then there is a positive constant $C= C(n, \Omega_0, \Omega_1, \Omega_2)$ with the following property. If $u$ is a continuous, convex function in $\overline{\Omega_2}$
with $u\leq 0$ on $\p\Omega_2$ then
$$\|u\|_{L^{\infty}(\Omega_1)}\leq C \int_{\Omega_2\setminus \Omega_0} |u| dx.$$
\end{lem}
\begin{proof}[Proof of Lemma \ref{clem}]
Suppose by contradiction that there exists a sequence of  continuous, convex functions $\{u_k\}$ in $\overline{\Omega_2}$
with $u_k\leq 0$ on $\p\Omega_2$ such that 
$$\|u_k\|_{L^{\infty}(\Omega_1)}=1\quad\text{but } \int_{\Omega_2\setminus \Omega_0} |u_k| dx \leq \frac{1}{k}.$$
Thus
$$  \int_{\Omega_2} |u_k| dx = \int_{\Omega_2\setminus \Omega_0} |u_k| dx +  \int_{ \Omega_0} |u_k| dx\leq \frac{1}{k} +|\Omega_0|.$$
Therefore, we have (see, for example, inequality (3.2) in \cite{Le18})
$$\|u_k\|_{L^{\infty}(\Omega_2)} \leq \frac{n+1}{|\Omega_2|}\int_{\Omega_2} |u_k| dx \leq C_1 (n,\Omega_0,\Omega_2).$$
From  $u_k\leq 0$ on $\p\Omega_2$ and the gradient bound for each $x\in\Omega_2$ (see, for example, (3.1) in \cite{Le18})
$$|Du_k(x)|\leq \frac{\max_{\p\Omega_2} u_k - u_k(x)}{\dist(x,\p\Omega_2)}\leq \frac{\|u_k\|_{L^{\infty}(\Omega_2)}}{\dist(x,\p\Omega_2)},$$
we find that, after extracting a subsequence, $\{u_k\}$ converges locally uniformly in $\Omega_2$ to a convex function $u$ in $\overline{\Omega_2}$ with $u\leq 0$ on $\p\Omega_2$. 
Hence
$\|u\|_{L^{\infty}(\Omega_1)}=1$. Moreover, from $$ \int_{\Omega_1\setminus \Omega_0} |u_k| dx \leq \frac{1}{k},$$ we find that $u\equiv 0$ in $\Omega_1\setminus \Omega_0$. 
By the convexity of $u$, we have $u\equiv 0$ in $\Omega_1$.  This contradicts $\|u\|_{L^{\infty}(\Omega_1)}=1$ and hence, the lemma is proved.
\end{proof}
\begin{cor}
Let $\Omega_0\subset\subset\subset \Omega$ be bounded, convex domains in $\R^n$ ($n\geq 2$). If $u$ is a continuous, convex function in $\overline{\Omega}$ then
\begin{equation}
\label{clem4}
\|u\|_{L^{\infty}(\Omega)} \leq C_1 (n,\Omega_0,\Omega,  \max_{\p\Omega} u) + C_2(n,\Omega_0,\Omega) \int_{\Omega\setminus\Omega_0}|u| dx.
\end{equation}
\end{cor}
\begin{proof}
Applying (3.2) in \cite{Le18} to $u-\max_{\p\Omega} u$, we get
\begin{equation}
\label{clem3}
\|u\|_{L^{\infty}(\Omega)}\leq C(n,\Omega)\int_{\Omega}|u| dx + C(n,\Omega, \max_{\p\Omega} u).
\end{equation}
Applying Lemma \ref{clem} to $u-\max_{\p\Omega} u$, we find
\begin{eqnarray*}\|u\|_{L^{\infty}(\Omega_0)} \leq \|u-\max_{\p\Omega} u\|_{L^{\infty}(\Omega_0)}+ |\max_{\p\Omega} u|&\leq& C(n,\Omega_0,\Omega) \int_{\Omega\setminus\Omega_0}|u-\max_{\p\Omega} u| dx + |\max_{\p\Omega} u|
\\ &\leq & C(n,\Omega_0,\Omega) \int_{\Omega\setminus\Omega_0}|u| dx+   C|\max_{\p\Omega} u|.
\end{eqnarray*}
It follows that
$$\int_{\Omega}|u| dx\leq |\Omega_0| \|u\|_{L^{\infty}(\Omega_0)} +  \int_{\Omega\setminus\Omega_0}|u| dx\leq C(n,\Omega_0,\Omega) \int_{\Omega\setminus\Omega_0}|u| dx+   C(n,\Omega_0,\Omega)|\max_{\p\Omega} u|$$
and therefore (\ref{clem4}) follows from (\ref{clem3}).
\end{proof}
\section{Proof of the main results }
\label{pf_sec}
In this section, we prove Theorems \ref{mthm} and \ref{ethm}.
\begin{proof}[Proof of Theorem \ref{mthm}] We divide the proof into several steps.\\
{\it Step 1: A priori estimates.}
In this step, we establish the a priori $L^{\infty}(\Omega)$ estimates for uniformly convex solutions $u_{\e}\in W^{4, q}(\Omega)$ $(q>n)$ to the system (\ref{Abe})-(\ref{fe}).  Recall that $\varphi\in W^{4,q}(\Omega)$ is convex. We only consider
$$0<\e<1.$$
\noindent
For $t>0$, let
 \begin{equation*}
 G(t)= \left\{ 
  \begin{alignedat}{1} \frac{t^{\theta}-1}{\theta}~&\text{if} ~\theta \in (0, 1/n), \\\
 \log t~&\text{if}~ \theta=0.
\end{alignedat}
\right.
\end{equation*}
Then $G'(t)= t^{\theta-1}$ for all $t>0$ and $w_\e= G^{'}(\det D^2 u_\e)$ in $\Omega$.\\

 In what follows, we use $C, C_0, C_1, C_2, \cdots,$ etc, to denote positive constants depending only on $n$, $q$, $\Omega_0,\Omega$, $\theta$, $C_\ast$, $\inf_{\p \Omega}\psi$, and $\| \varphi\|_{W^{4, q}(\Omega)}$. They are called universal constants and their values may change from line to line. However, they do not depend on $\e>0$. 
 {\it When such constants depend on $\e$, they will be indicated explicitly, as for $C_d(\varphi,\e)$ below.}

Recall from (\ref{rhoeq}) that $\rho$ is a strictly convex defining function of $\Omega$. Then, there is $\gamma>0$ depending only on $\Omega$ such that
$$D^2 \rho\geq \gamma I_n~ \text{and} ~\rho\geq -\gamma^{-1} ~\text{in}~ \Omega.$$
Recall that the constant $C_\varphi$ is defined by (\ref{Cvar}). For simplicity, let us denote
\begin{equation}
\label{tildeu}
\tilde u =\varphi + C_{\varphi}\e^{\frac{1}{3n^2}} (e^{\rho}-1).
\end{equation}
From the convexity of $\varphi$ and 
$$D^2 (e^{\rho}-1)=e^{\rho}(D^2\rho + D\rho \otimes D\rho)\geq e^{-\gamma^{-1}}\gamma I_n,$$
we find that the  function $\tilde u$
is uniformly convex, belongs to $ W^{4,q}(\Omega)$ and satisfies:
 \begin{myindentpar}{1cm}
  (a) $\tilde u=\varphi$ on $\p\Omega$,\\
  (b) $
\| \tilde{u} \|_{C^{3}(\ov{\Omega})} + 
\|\tilde u\|_{W^{4, q}(\Omega)} \leq C,\quad \textrm{and } \det D^2\tilde{u} \ge C^{-1}C_{d}(\varphi, \e)>0,
$\\
(c) letting $\tilde{w}=G'(\det D^2 \tilde{u})$, and denoting by $(\tilde{U}^{ij})$ the 
cofactor matrix of $(\tilde{u}_{ij})$, then  $$\|\tilde w\|_{L^{\infty}(\Omega)}\leq C[C_d(\varphi,\e)]^{-1};~\left\|\tilde U^{ij}\tilde w_{ij}\right\|_{L^1(\Omega)}\leq C [C_d(\varphi,\e)]^{-3}.$$
\end{myindentpar}
Here, from the definition of $\tilde u$ in (\ref{tildeu}), we have the following estimate for the magnitude of $\det D^2 \tilde u$ in terms of $\e$:
\begin{equation}
\label{Cd}
C_{d}(\varphi,\e) \\= \left\{\begin{array}{rl}
 \min_{\overline{\Omega}} \det D^2\varphi&  \text{if }\varphi \text{ is uniformly convex in } \overline{\Omega},\\
\e^{\frac{1}{3n}} & \text{otherwise}.
\end{array}\right.
\end{equation}
Note that (c) follows from (b) and the following formula (see also \cite[Lemma 2.1]{Le_JDE}):
$$\tilde w_{ij}= G^{'''} (\det D^2 \tilde{u}) \tilde U^{kl} \tilde U^{rs} \tilde u_{kli}\tilde u_{rsj}+ G^{''} (\det D^2 \tilde{u}) \tilde U^{kl} \tilde u_{klij}+ G^{''} (\det D^2 \tilde{u})\tilde U^{kl}_j \tilde u_{kli}.$$
 
We use $\nu= (\nu_{1},\cdots,\nu_n)$ to denote the unit outer normal vector field on $\p \Omega$ and $\nu_0$ on $\p\Omega_0$. \\
\noindent
 First, from (4.5) in \cite{Le18}, we have
 \begin{equation}
 \label{tildew}
 \int_{\partial \Omega} (\psi U_\e^{ij}-\tilde w\tilde U^{ij}) ((u_\e)_j-\tilde{u}_j ) \nu_i dS+ \int_{\Omega} U^{ij}_\e (w_{\e})_{ij}(u_\e-\tilde{u})dx+ \int_{\Omega} \tilde U^{ij} \tilde w_{ij} (\tilde{u}-u_\e) dx\leq 0.
 \end{equation}
 For readers' convenience, we include the derivation of (\ref{tildew}) which relies on some concavity arguments. Indeed, from the definition of $G$ and $\theta\in [0,\frac{1}{n})$, we find that the function $\tilde G(t):= G(t^n)$ is strictly concave on$(0,\infty)$.
 Using this together with $G'>0$, and the concavity of the map $M\longmapsto (\det M)^{1/n}$ in the space of symmetric matrices $M\geq 0$, we obtain
\begin{multline*}
 \tilde G((\det D^2 \tilde u)^{1/n}) -\tilde G((\det D^2  u_\e)^{1/n})\\ \leq \tilde G^{'}((\det D^2 u_\e)^{1/n})((\det D^2 \tilde u)^{1/n}-(\det D^2 u_\e)^{1/n})\\
 \leq  \tilde G^{'}((\det D^2 u_\e)^{1/n})\frac{1}{n} (\det D^2 u_\e)^{1/n-1} U_\e^{ij} (\tilde u-u_\e)_{ij} .
\end{multline*}
Since $\tilde G^{'}((\det D^2 u_\e)^{1/n}) = n G^{'}(\det D^2 u_\e) (\det D^2 u_\e)^{\frac{n-1}{n}}= n w_\e  (\det D^2 u_\e)^{\frac{n-1}{n}}$, we rewrite the above inequalities as
\begin{equation*}
G(\det D^2 \tilde u)- G(\det D^2 u_\e)\leq w_\e U_\e^{ij}(\tilde u- u_\e)_{ij}.
\end{equation*}
Similarly, we have
\begin{equation*}
G(\det D^2 u_\e)- G(\det D^2 \tilde u)\leq \tilde w \tilde U^{ij}(u_\e- \tilde u)_{ij}.
\end{equation*}
Adding these two last inequalities, integrating by parts twice and using the fact that $(U_\e^{ij})$ is divergence free, we obtain
\begin{eqnarray*}
0&\geq&  \int_{\Omega} \left[w_\e U_\e^{ij}(u_\e-\tilde u)_{ij} + \tilde w \tilde U^{ij}(\tilde u-u_\e)_{ij}\right]dx
\nonumber \\
&=&
  \int_{\partial \Omega} w_\e U_\e^{ij} ((u_\e)_j-\tilde{u}_j ) \nu_i dS + \int_{\Omega} U_\e^{ij} (w_\e)_{ij}(u_\e-\tilde{u}) dx  \\ &+&
 \int_{\partial \Omega} \tilde w \tilde{U}^{ij} ( \tilde{u}_j-(u_\e)_j) \nu_i dS + \int_{\Omega} \tilde U^{ij} \tilde w_{ij} (\tilde{u}-u_\e) dx,
  \end{eqnarray*}
  from which (\ref{tildew})  follows. Here we recall that $w_\e=\psi$ on $\Omega$.
 
 In what follows, we will use $f_\e$ to denote $f_\e(\cdot, u_\e, Du_\e, D^2 u_\e; \tilde u)$. Then, by (\ref{Abe}),
 $$U^{ij}_\e (w_{\e})_{ij} =\e^{-1} f_\e.$$
 Let $K(y)$ be the Gauss curvature of $\partial  \Omega$ at
$y\in\p\Omega$. We have the following assertion.\\
{\bf Assertion.}  For all $\e>0$, we have
\begin{equation*}
 \int_{\partial \Omega} K \psi (u_\e)_{\nu}^n dS \leq  C[C_d(\varphi,\e)]^{-3} + C[C_d(\varphi,\e)]^{-3}\left(
\int_{\p \Omega}  ((u_\e)_{\nu}^+)^n dS\right)^{(n-1)/n} +  \int_{\Omega} -\e^{-1}f_\e (u_\e-\tilde u)dx.
\end{equation*}
The proof of the assertion is similar to that of (4.10) in \cite{Le18}, and we include it here for readers' convenience.

We start by analyzing the boundary terms in (\ref{tildew}). Since $u_\e-\tilde u=0$ on $\p \Omega$, we have $(u_\e-\tilde u)_j= (u_\e-\tilde u)_{\nu} \nu_j$, and hence
$$U_\e^{ij}(u_\e-\tilde u)_j \nu_i= U_\e^{ij}\nu_j \nu_i (u_\e-\tilde u)_{\nu} = U_\e^{\nu\nu}(u_\e-\tilde u)_{\nu}$$
where
$$U_\e^{\nu\nu} = \det D^2_{x'} u_\e$$
with $x'\perp \nu$ denoting the tangential directions along $\p \Omega$. Therefore, 
\begin{equation*} 
(\psi U_\e^{ij}-\tilde w\tilde U^{ij}) ((u_\e)_j - \tilde{u}_j) \nu_i = (\psi U_\e^{\nu\nu} -\tilde w\tilde U^{\nu\nu})((u_\e)_{\nu} - \tilde{u}_{\nu}).
\end{equation*}
Now, using $U^{\nu\nu}_\e$ and $\tilde U^{\nu\nu}$, we can rewrite (\ref{tildew}) as
$$\int_{\p\Omega}  (\psi U_\e^{\nu\nu} -\tilde w\tilde U^{\nu\nu})((u_\e)_{\nu} - \tilde{u}_{\nu}) dS +  \int_{\Omega} \tilde U^{ij} \tilde w_{ij} (\tilde{u}-u_\e) dx \leq  \int_{\Omega} -\e^{-1}f_\e (u_\e-\tilde u)dx$$
which gives
\begin{multline}
\label{KUep}
\int_{\p\Omega} \psi U_\e^{\nu\nu} (u_\e)_{\nu} dS \leq  \|\psi \tilde{u}_{\nu}\|_{L^{\infty}(\Omega)} \int_{\p\Omega}  |U_\e^{\nu\nu} |dS +  \|\tilde U^{\nu\nu} \tilde{u}_{\nu}\|_{L^{\infty}(\Omega)}   \|\tilde w\|_{L^{\infty}(\Omega)} 
\\ +  \|\tilde U^{\nu\nu} \|_{L^{\infty}(\Omega)}   \|\tilde w\|_{L^{\infty}(\Omega)} \int_{\p\Omega}|(u_\e)_\nu| dS
 +  \|\tilde U^{ij}\tilde w_{ij}\|_{L^1(\Omega)}\left(\|\tilde u\|_{L^{\infty}(\Omega)} + \| u_\e\|_{L^{\infty}(\Omega)}\right)\\+ \int_{\Omega} -\e^{-1}f_\e (u_\e-\tilde u)dx.
\end{multline}
Observe that:
\begin{myindentpar}{1cm}
(A) By (a)-(b), the quantities $\tilde u$, $\tilde{u}_{\nu}$, and $\tilde{U}^{\nu\nu}=\tilde U^{ij}\nu_i\nu_j$ are universally bounded. By (c), 
 $$\|\tilde w\|_{L^{\infty}(\Omega)} \leq C[C_d(\varphi,\e)]^{-1}; \|\tilde U^{ij}\tilde w_{ij}\|_{L^1(\Omega)} \leq C[C_d(\varphi,\e)]^{-3}.$$
 (B) For the convex function $u_\e\in C^2(\overline{\Omega})$ with $u_\e=\varphi$ on $\p\Omega$, we have (see, for example \cite[inequality (2.7)]{Le_JDE})
\begin{equation*}\|u_\e\|_{L^{\infty}(\Omega)} \leq C(n,\varphi,\Omega) + C(n,\Omega)\left(
\int_{\p \Omega}  ((u_\e)_{\nu}^+)^n dS\right)^{1/n}.
\end{equation*}
(C) The Gauss curvature $K$ of $\p\Omega$ and $U^{\nu\nu}_\e$ are related by (see, for example, (4.9) in \cite{Le18})
\begin{equation*} 
U^{\nu\nu}_\e = K ((u_\e)_\nu)^{n-1} + E_\e\quad\text{where } |E_\e| \leq C(1 + ((u_\e)^{+}_\nu)^{n-2}).
\end{equation*}
 \end{myindentpar}
Now, the Assertion follows from (\ref{KUep}) together with the above observations. 

Let us continue with the proof of {\it Step 1}. Since $u_\e$ is convex with boundary value $\varphi$ on $\p\Omega$, we have 
 \begin{equation*}
 (u_\e)_\nu\geq -\|D\varphi\|_{L^{\infty}(\Omega)}:=-C_0.
 \end{equation*}
It follows that, for $u_{\nu}^+ =\max (0, u_{\nu})$, we have 
$((u_\e)_\nu^{+})^n \leq (u_\e)_{\nu}^n + C_0^n$
and therefore from the Assertion, we obtain
\begin{multline}
\label{fe_u}
 \int_{\partial \Omega} K \psi ((u_\e)_{\nu}^{+})^n dS \leq  C[C_d(\varphi,\e)]^{-3} + C[C_d(\varphi,\e)]^{-3}\left(
\int_{\p \Omega}  ((u_\e)_{\nu}^+)^n dS \right)^{(n-1)/n}\\ +  \int_{\Omega} -\e^{-1}f_\e (u_\e-\tilde u)dx.
\end{multline}
By the uniform convexity of $\p\Omega$, we have $K\geq C(\Omega)>0$ on $\p\Omega$. Using this, together with $\inf_{\p \Omega}\psi>0$ and Young's inequality for the second term on the right hand side of (\ref{fe_u}), we find that
$$C[C_d(\varphi,\e)]^{-3}\left(
\int_{\p \Omega}  ((u_\e)_{\nu}^+)^n dS \right)^{(n-1)/n} \leq C[C_d(\varphi,\e)]^{-3n} +\frac{1}{2} \int_{\partial \Omega} K \psi ((u_\e)_{\nu}^{+})^n dS$$
and hence
\begin{equation*}
 \int_{\partial \Omega}  ((u_\e)_{\nu}^{+})^n dS \leq  C[C_d(\varphi,\e)]^{-3n}  +  \int_{\Omega} -\e^{-1}f_\e (u_\e-\tilde u)dx.
\end{equation*}
Therefore, multiplying both sides of the above inequality by $\e>0$ and using $\e [C_d(\varphi,\e)]^{-3n}\leq C$ from (\ref{Cd}), we get, as in  inequality (4.31) in \cite{Le18} (which was stated for $n=2$ there)
\begin{equation}
\label{fe_u2}
\int_{\partial \Omega} \e((u_\e)_{\nu}^{+})^n dS \leq  C\e [C_d(\varphi,\e)]^{-3n} +  \int_{\Omega} -f_\e (u_\e-\tilde u)dx\leq C+  \int_{\Omega} -f_\e (u_\e-\tilde u)dx.
\end{equation}
We will  estimate the right hand side of (\ref{fe_u2}). \\
From the convexity of $F^0$ (see (\ref{F0})), we can estimate
\begin{equation}
\label{Ae}
A_\e:=\int_{\Omega_0}- f^0(x, u_\e(x)) (u_\e-\tilde u) dx\leq \int_{\Omega_0}- f^0(x, \tilde u(x)) (u_\e-\tilde u) dx\leq  C +  C_1\|u_\e\|_{L^{\infty}(\Omega_0)}.
\end{equation}
In what follows, we will frequently use the following inequality (see, (3.1) in \cite{Le18})
 \begin{equation}
 \label{Due}
 |Du_\e(x)|\leq \frac{\max_{\p\Omega} u_\e - u_\e(x)}{\dist(x,\p\Omega)}~\forall x\in\Omega.
 \end{equation}
 By the convexity of $u_\e$ and $F^1(x, p)$ in $p$, we have $F^{1}_{p_i p_j} (u_\e)_{ij}\geq 0$. Moreover, $u_\e\leq \sup_{\p\Omega}\varphi\leq C$ and $|\tilde u|\leq C$. Thus, recalling (\ref{F1}), we find that
 \begin{equation*}
F^{1}_{p_i p_j} (u_\e)_{ij} (u_\e-\tilde u) \leq C F^{1}_{p_i p_j} (u_\e)_{ij} \leq C C_\ast \Delta u_\e.
 \end{equation*}
 By the divergence theorem and (\ref{Due}), we have
 \begin{equation}
  \label{F1_est1}
  \int_{\Omega_0} F^{1}_{p_i p_j} (u_\e)_{ij} (u_\e-\tilde u) dx \leq C C_\ast \int_{\Omega_0}  \Delta u_\e dx = CC_\ast \int_{\p \Omega_0} (u_\e)_{\nu_0} dS \leq C + C_2 \|u_\e\|_{L^{\infty}(\Omega_0)}.
 \end{equation}
 
 On the other hand, for any $i=1,\cdots, n$, using (\ref{F1}) and (\ref{Due}), we can estimate in $\Omega_0$:
 \begin{eqnarray}
 \label{F1_est2}
 |F^1_{p_i x_i} (x, Du_\e(x)) (u_\e(x)-\tilde u(x) )| &\leq& C_\ast ( |Du_\e(x)|+ 1)(|u_\e(x) + C) 
\nonumber\\ &\leq& C_3(|u_\e(x)|^2 + 1).
\end{eqnarray}
Note that (\ref{clem4}) together with $\|\tilde u\|_{L^{\infty}(\Omega)}\leq C$ gives
\begin{equation}
\label{clem2}
\|u_\e\|_{L^{\infty}(\Omega_0)}\leq C +  C\int_{\Omega\setminus\Omega_0}|u_\e| dx\leq C +   C_4\int_{\Omega\setminus\Omega_0}|u_\e-\tilde u|^2 dx. 
\end{equation}
From (\ref{Ae}), (\ref{F1_est1}), (\ref{F1_est2}) and (\ref{clem2}), we find that
\begin{eqnarray}
\label{bounde1}
\int_{\Omega_0}- f_{\e} (u_\e-\tilde u) dx &=& \int_{\Omega_0}\left[- f^0(x, u_\e(x)) + \frac{\p}{\p x_i} \left(\frac{\p F^1}{\p p_i}(x, Du_\e(x))\right)\right] (u_\e-\tilde u)dx\nonumber\\
&=& A_\e + \int_{\Omega_0} \left [F^1_{p_i x_i} (x, Du_\e(x)) + F^1_{p_i p_j}(x, Du_\e(x)) (u_\e)_{ij}\right] (u_\e-\tilde u) dx\nonumber\\
&\leq & A_\e + C_3\int_{\Omega_0}  (|u_\e|^2 + 1) dx+ C + C_2 \|u_\e\|_{L^{\infty}(\Omega_0)} \nonumber\\
&\leq& C_5\|u\|^2_{L^{\infty}(\Omega_0)} +C
\leq C + C_6 \int_{\Omega\setminus\Omega_0}|u_\e-\tilde u|^2 dx.
\end{eqnarray}
It follows from (\ref{fe_u2}) and (\ref{clem2}), and $f_\e=\frac{1}{\e} (u_\e-\tilde u)$ on $\Omega\setminus\Omega_0$ that
\begin{eqnarray} 
\label{ep_1}
\int_{\partial \Omega} \e  ((u_\e)^{+}_{\nu})^n dS& \le& C  +  \int_{\Omega} -f_{\e} (u_\e-\tilde u ) dx\nonumber \\&=& C + \int_{\Omega_0}- f_{\e} (u_\e-\tilde u ) dx+ \int_{\Omega\setminus\Omega_0} -f_{\e} (u_\e-\tilde u) dx\nonumber\\
 &\leq& C + C_6 \int_{\Omega\setminus\Omega_0}|u_\e-\tilde u|^2 dx   +  \int_{\Omega\setminus\Omega_0} -\frac{1}{\e}|u_\e-\tilde u|^2 dx\nonumber
 \\&\leq& C  -  \int_{\Omega\setminus\Omega_0} \frac{1}{2\e}|u_\e-\tilde u|^2 dx
\end{eqnarray} 
if $\e$ is small, say $$\e\leq \frac{1}{2C_6}.$$  {\it From now on,  we assume that $\e$ is small.} Then, 
we get
 \begin{equation}
 \label{epn}
 \int_{\partial \Omega} \e^2  ((u_\e)^{+}_{\nu})^n  dS+ \int_{\Omega\setminus\Omega_0} (u_\e-\tilde u)^2 dx  \leq C\e.
 \end{equation}
This together with (\ref{clem4}) and $\|\tilde u\|_{L^{\infty}(\Omega)}\leq C$ gives the uniform bound for $u_\e$ on $\Omega$:
\begin{equation}
\label{ueb}
\|u_\e\|_{L^{\infty}(\Omega)} \leq C_7.
\end{equation}
{\it Step 2: Existence and convergence properties of uniformly convex solutions to (\ref{Abe})-(\ref{fe}).}\\
(i) We consider two separate cases.\\
{\it Case 1:  $F(x, z, p)= F^0(x, z)$}. In this case, from the a priori $L^{\infty}(\Omega)$ estimates (\ref{ueb}) for uniformly convex solutions $u_{\e}\in W^{4, q}(\Omega)$ $(q>n)$ to the system (\ref{Abe})-(\ref{fe}), 
we can use a Leray-Schauder degree argument as in \cite[Theorem 4.2]{CR} to show the existence of a unique uniformly convex solution  $u_{\e}\in W^{4, q}(\Omega)$ (for all $q<\infty$) to the system (\ref{Abe})-(\ref{fe}).\\
\noindent
{\it Case 2:  $F(x, z, p)= F^0(x, z)+ F^1 (x, p)$ and $n=2$}. In this case, from the a priori $L^{\infty}(\Omega)$ estimates (\ref{ueb}) for uniformly convex solutions $u_{\e}\in W^{4, q}(\Omega)$ $(q>n)$ to the system (\ref{Abe})-(\ref{fe}), 
we can establish the a priori $W^{4, q}(\Omega)$ estimates for $u_\e$ as in \cite[Theorem 4.1]{Le18}. With these a priori estimates,
we can use a Leray-Schauder degree argument as in \cite[Theorem 2.1]{Le18} to show the existence of a uniformly convex solution  $u_{\e}\in W^{4, q}(\Omega)$ (for all $q<\infty$)  to the system (\ref{Abe})-(\ref{fe}).\\
 Hence (i) is proved. \\
 (ii) For $\e>0$ small, let $u_\e\in W^{4, q}(\Omega)$ $(q>n)$ be a solution to (\ref{Abe})-(\ref{fe}).
 By (\ref{ueb}), the sequence $\{u_\e\}$ is uniformly bounded with respect to $\e$. By (\ref{Due}), $|Du_\e|$ is uniformly bounded on compact subsets of $\Omega$. Thus, by the Arzela--Ascoli  theorem, up to extraction of a subsequence, 
  $u_\e$ converges uniformly on compact subsets of $\Omega$, and also in $W^{1,2}(\Omega_0)$, to a convex function $u$ on $\Omega$. From (\ref{epn}) and the fact that $\lim_{\e\rightarrow 0} \tilde u=\varphi$, we find $u\in \bar S[\varphi,\Omega_0]$.  Let 
 \begin{equation}
 \label{eta_def}
 \eta_\e:= \e^{1/n} \left(\int_{\partial \Omega}   ((u_\e)^{+}_{\nu})^n dS\right)^{1/n}.
 \end{equation}
 Then, from (\ref{epn}), we have as in \cite[inequality (5.5)]{CR} and \cite[inequality (4.27)]{Le18}
 \begin{equation}
 \label{etae}
 \eta_\e \leq C.
 \end{equation}
 Consider the functional $J$ defined over $\bar S[\varphi, \Omega_0]$ by 
\begin{equation}
\label{J_fn}
J(v):=\int_{\Omega_0} F(x, v(x), Dv(x)) dx.
\end{equation}
Since $u_\e$ converges uniformly to $u$ on $\overline{\Omega_0}$, by Fatou's lemma, we have
\begin{equation*}
\liminf_{\e}\int_{\Omega_0}F^0(x, u_\e(x)) dx\geq \int_{\Omega_0} F^0(x, u(x)) dx.
\end{equation*}
 From the convexity of $F^1(x, p)$ in $p$ and the fact that $u_\e$ converges to $u$ on $W^{1,2}(\Omega_0)$, we have
 \begin{equation*}
 \liminf_{\e}\int_{\Omega_0}F^1(x, Du_\e(x)) dx\geq \int_{\Omega_0} F^1(x, Du(x)) dx,
 \end{equation*}
 which is due to lower semicontinuity.
 Therefore
 \begin{equation}
 \label{ueu_est}
 \liminf_{\e}J(u_\e)\geq J(u).
 \end{equation}
Our main estimate is the following.\\
 {\bf Claim.} If $0\leq\theta<1/n$, then for any $v\in \bar S[\varphi, \Omega_0]$, we have
\begin{equation}
\label{vue0}
J(v)\geq  \liminf_{\e}J(u_\e)- \limsup_{\e} \left[\e^{(n-1)/n}\eta_\e +  \e^{\frac{1}{n}}\eta^{n-1}_\e\right].
\end{equation}
Assuming the above claim, we show that $u$ is a minimizer of (\ref{p1}). Indeed, this follows from (\ref{vue0}), (\ref{ueu_est}) and 
(\ref{etae}) which imply the estimate
$
J(v)\geq J(u)
$
for all $v\in \bar S[\varphi, \Omega_0]$.
 
It remains to prove the claim.  The proof is similar to that of  \cite[Theorem 2.3]{Le18} where the case $n=2$ was treated. In our context of Theorem \ref{mthm} (ii), we would like to treat also the case of general dimensions $n$ when $F^1\equiv 0$, that is, when the Lagrangian is independent of the gradient variables.
For reader's convenience, we repeat the arguments there. Recall from (\ref{tildeu}) that 
$$\tilde u=\varphi + C_{\varphi}\e^{\frac{1}{3n^2}} (e^{\rho}-1).$$
Consider the following functional $J_\e$ over the set of convex functions $v$ on $\overline{\Omega}$:
\begin{equation}
\label{Je_def}
J_{\e}(v)=\int_{\Omega_0}  F(x, v(x), Dv(x))dx+\frac{1}{2\e}\int_{\Omega\setminus\Omega_0} (v-\tilde u)^2 dx-\e\int_{\Omega} G( \det D^2 v) dx.
\end{equation}
From the Alexandrov theorem (\cite[Theorem 1, p.242]{EG}), $v$ is twice differentiable a.e.  At those points of twice differentiability of $v$, we use $D^2 v$ to denote
its Hessian matrix. Thus, in addition to setting $\log 0=-\infty$, the functional $J_\e$ is well defined with this convention for all $\theta\geq 0$; it can take value $\infty$.

Let $U^{\nu\nu}_\e=U^{ij}_\e \nu_i\nu_j$.  Let $K$ be the Gauss curvature of $\p\Omega$. Then, we have (see, for example, (4.9) in \cite{Le18})
\begin{equation} \label{KE}
U^{\nu\nu}_\e = K ((u_\e)_\nu)^{n-1} + E_\e\quad\text{where } |E_\e| \leq C(1 + ((u_\e)^{+}_\nu)^{n-2}).
\end{equation}
First, by \cite[estimate (5.6)]{Le18}, if $v$ is a convex function in $\overline{\Omega}$ with $v=\tilde u$ in a neighborhood of $\p\Omega$,  then
\begin{equation}
\label{Je}
J_\e(v)-J_\e(u_\e)\geq   \e \int_{\p\Omega} \psi U_{\e}^{\nu\nu} \p_{\nu}(u_\e-\tilde u) dS+  \int_{\p\Omega_0}(v-u_\e)\nabla_p F^1(x, Du_\e(x)) \cdot \nu_0 dS.
\end{equation} 
Now, we are ready to prove (\ref{vue0}) for all $v\in\bar S[\varphi,\Omega_0]$. Indeed, applying (\ref{Je}) to
$$v_\e: = v + C_{\varphi}\e^{\frac{1}{3n^2}}(e^\rho-1),$$
which clearly satisfies $v_\e=\tilde u$ on $\overline{\Omega}\setminus \Omega_0$,
and using the fact that the subsequential uniform limit $u\in \bar S[\varphi, \Omega_0]$ of $u_\e$ satisfies $u=v=\varphi$ on $\p\Omega_0$, we   conclude that
the corresponding rightmost term in  (\ref{Je}) $$ \int_{\p\Omega_0}(v_\e-u_\e)\nabla_p F^1(x, Du_\e(x)) \cdot \nu_0 dS\rightarrow 0~\text{as }\e\rightarrow 0,$$  and hence, 
\begin{multline}
\label{J_comp}
 \e \int_{\p\Omega} \psi U_{\e}^{\nu\nu} \p_{\nu}(u_\e-\tilde u)dS-o_{\e}(1)\leq J_\e(v_\e)-J_\e(u_\e) \\
=J(v_\e)- J(u_\e) -\frac{1}{2\e}\int_{\Omega\setminus\Omega_0} (u_\e-\tilde u)^2 dx \\-\e\int_{\Omega} [G(\det D^2 v_\e)- G(\det D^2 u_\e)] dx .
\end{multline}
Here we use $o_{\e}(1)$ to denote a quantity that tends to $0$ when $\e\rightarrow 0$.\\
\noindent
Since $\det D^2 v_\e \geq C\e^{\frac{1}{3n}}$, 
we have
\begin{equation}
\label{Gpos}
\liminf_{\e} \e\int_{\Omega} G(\det D^2 v_\e)dx\geq 0~\text{for all } v\in \bar S[\varphi,\Omega_0].
\end{equation}
From the definition of $v_\e$ and the dominated convergence theorem, we have $$\lim_{\e \rightarrow 0} J(v_\e)=J(v).$$ Combining this with  (\ref{J_comp}) and (\ref{Gpos}), we get
\begin{multline}
\label{Jv}
J(v)\geq \liminf_{\e} J(u_\e) + \liminf_{\e} \e\int_{\Omega} [G(\det D^2 v_\e)- G(\det D^2 u_\e)] dx\\+ \liminf_{\e}\e \int_{\p\Omega} \psi U_{\e}^{\nu\nu} \p_{\nu}(u_\e-\tilde u)dS
\\ \geq \liminf_{\e} J(u_\e) - \limsup_{\e} \e\int_{\Omega} G(\det D^2 u_\e)dx+ \liminf_{\e}\e \int_{\p\Omega} \psi U_{\e}^{\nu\nu} \p_{\nu}(u_\e-\tilde u)dS.
\end{multline}
From  (\ref{KE}), $\|\tilde u_{\nu}\|_{L^{\infty}(\p\Omega)}\leq C $ and $(u_\e)_\nu \geq (u_\e)^{+}_\nu-C_0$ by the estimate preceding (\ref{fe_u}), we have
\begin{equation*}
 U_{\e}^{\nu\nu} \p_{\nu}(u_\e-\tilde u) \geq -C ((u_\e)^{+}_\nu)^{n-1} -C.
\end{equation*}
From the definition of $\eta_\e$ in (\ref{eta_def}), one has
$$\int_{\p\Omega} (u_\e)^{+}_\nu dS\leq C\e^{-1/n}\eta_\e\quad\text{and }\int_{\p\Omega} ((u_\e)^{+}_\nu)^{n-1} dS\leq C\e^{-(n-1)/n}\eta^{n-1}_\e$$
and hence, 
\begin{equation}
\label{ubdr}
 \e\int_{\p\Omega} \psi U_{\e}^{\nu\nu} \p_{\nu}(u_\e-\tilde u)dS \geq -C\e \int_{\p\Omega}[1 + ((u_\e)^{+}_\nu)^{n-1}]dS\geq -C \e^{\frac{1}{n}}\eta^{n-1}_\e.
\end{equation}
From $0\leq \theta<1/n$ and the convexity of $u_\e$, we can find $C>0$ depending only on $\theta$ and $n$ such that
$$G(\det D^2 u_\e) \leq C [1+ (\det D^2 u_\e)^{1/n}]  \leq C (1 + \Delta u_\e).$$
Therefore, from the divergence theorem, we obtain
\begin{equation}
\int_{\Omega}G(\det D^2 u_\e)dx\leq  C\int_{\Omega} (1 + \Delta u_\e) dx=
C|\Omega| + C\int_{\p\Omega} (u_\e)_\nu dS \leq C(1 +\e^{-1/n}\eta_\e).
\label{udet}
\end{equation}
Combining (\ref{Jv})--(\ref{udet}), we get
$$J(v)\geq \liminf_{\e}J(u_\e) - \limsup_{\e} \left[\e^{(n-1)/n}\eta_\e +  \e^{\frac{1}{n}}\eta^{n-1}_\e\right]$$
which implies (\ref{vue0}). The Claim is proved and the proof of the theorem is complete.
\end{proof}
\begin{proof}[Sketch of proof of Theorem \ref{ethm}] 
The proof is parallel to that of Theorem \ref{mthm} with some minor modifications. We briefly indicate these. Recall that $n=2$. \\
(i) {\it Existence result.} The key is still the a priori estimate  (\ref{ueb}) for a uniformly convex solution $u_\e \in W^{4,q}(\Omega)$ $(q>n)$ to (\ref{Abe})-(\ref{fe}).
Assume that there holds the linear growth condition of $f^0(x, z)$ with respect to $z$ in (\ref{F0l}).  In this case, the quantity $A_\e$ define in (\ref{Ae}) can be estimated from above by $$A_\e \leq C + C_1 \|u_\e\|^2_{L^{\infty}(\Omega_0)}.$$ Thus, the final estimate in (\ref{bounde1}) holds. 
Therefore (\ref{ueb}) holds if $\e$ is small; the constant $C_7$ now depends also on $C_l$.
As a consequence, the existence result of Theorem \ref{ethm} (i) follows as that of Theorem \ref{mthm}(i). \\
(ii) {\it Convergence result.} The new input here is the following well-known trace inequality. There is a constant $C_t=C_t(\Omega_0)>0$ depending only on $\Omega_0$ such that 
\begin{equation}
\label{tineq}
\int_{\Omega_0} |v- u|^2dx \leq C_t \int_{\Omega_0}|Dv- Du|^2dx + C_t \int_{\p\Omega_0} |v-u|^2 dS\quad \text{for all }u, v\in W^{1,2}(\Omega_0).
\end{equation}
Assume that (\ref{F0l}) and (\ref{F1s}) hold. With (\ref{F0l}), we have 
\begin{equation}
\label{F01}
F^0(x, \tilde z)- F^0(x, z)\geq f^0(x, z) (\tilde z-z) - \frac{C_b}{2}|\tilde z- z|^2\quad \text{for all }x\in\Omega_0 \text{ and all  }z, \tilde z\in\R.
\end{equation}
With (\ref{F1s}), we have 
\begin{equation}
\label{F11}
F^1(x, \tilde p)- F^1(x, p)\geq \nabla_p F^1(x, p) \cdot(\tilde p-p) + \frac{\underbar C}{2} |\tilde p- p|^2\quad \text{ for all }x\in\Omega_0 \text{ and all  }p, \tilde p\in\R^2.
\end{equation}
For $\e>0$ small, let $u_\e\in W^{4, q}(\Omega)$ $(q>n)$ be a solution to (\ref{Abe})-(\ref{fe}). \\

{\it Step 1: Convergence of a subsequence of $\{u_\e\}$ to a minimizer of (\ref{p1}).}
As in the proof of Theorem \ref{mthm} (ii), up to extraction of a subsequence, $u_\e$ converges uniformly on compact subsets of $\Omega$, and also in $W^{1,2}(\Omega_0)$, to a convex function $u\in \overline{S}[\varphi,\Omega_0]$. 

We show that $u$ is a minimizer of (\ref{p1}).  
The proof is similar to that of Theorem \ref{mthm}(ii)
except that (\ref{Je}) is replaced by 
 \begin{multline}
\label{Je2}
J_\e(v)-J_\e(u_\e)\geq   \e \int_{\p\Omega} \psi U_{\e}^{\nu\nu} \p_{\nu}(u_\e-\tilde u) dS+  \int_{\p\Omega_0}(v-u_\e)\nabla_p F^1(x, Du_\e(x)) \cdot \nu_0 dS\\
- \frac{\underbar C}{2}\int_{\p \Omega_0} |v- u_\e|^2 dS
\end{multline} 
for all convex functions $v$  in $\overline{\Omega}$ with $v=\tilde u$ in a neighborhood of $\p\Omega$
and provided that
\begin{equation}
\label{bC}
\underbar C\geq C_t C_b + 1.
\end{equation}
In (\ref{Je2}), the function $\tilde u$ is defined as in (\ref{tildeu}).
Clearly, when $v\in \overline{S}[\varphi, \Omega_0]$, the extra boundary term in (\ref{Je2}) disappears in the limit $\e\rightarrow 0$.

Now, we explain how to obtain (\ref{Je2}) from (\ref{F0l}), (\ref{F1s}) and (\ref{bC}). \\
Again, let $v$ be a convex function in $\overline{\Omega}$ with $v=\tilde u$ in a neighborhood of $\p\Omega$.
In the derivation of (\ref{Je}) in \cite[estimate (5.6)]{Le18}, we used (\ref{F01}) with $C_b=0$ (and $\tilde z$ a mollification $v_h$ of $v$ and $z$ the function $u_\e$) and (\ref{F11}) with $\underbar C=0$ (and $\tilde p$ the gradient $Dv_h$ and $p$ the gradient $Du_\e$); see  \cite[estimate (5.10)]{Le18}. With $C_b>0, \underbar C>0$, instead of (\ref{Je}), we have the following:
\begin{multline}
\label{Je1}
J_\e(v)-J_\e(u_\e)\geq   \e \int_{\p\Omega} \psi U_{\e}^{\nu\nu} \p_{\nu}(u_\e-\tilde u) dS+  \int_{\p\Omega_0}(v-u_\e)\nabla_p F^1(x, Du_\e(x)) \cdot \nu_0 dS\\
- \frac{C_b}{2}\int_{\Omega_0} |v- u_\e|^2 dx + \frac{\underbar C}{2} \int_{\Omega_0}|Dv- Du_\e|^2 dx.
\end{multline} 
Thus, provided (\ref{bC}) holds, (\ref{Je2}) follows from (\ref{Je1}) and (\ref{tineq}). \\

{\it Step 2: The whole sequence $\{u_\e\}$ converges to the unique minimizer in $\bar{S}[\varphi,\Omega_0]$  of (\ref{p1}) when (\ref{bC}) holds.} To show this, in view of {\it Step 1}, it suffices to show that (\ref{p1}) has  unique minimizer in $\bar{S}[\varphi,\Omega_0]$. 

Suppose that $u, v\in \bar{S}[\varphi,\Omega_0]$ are two minimizers of the functional $J$ defined by
$$J(u):=\int_{\Omega_0} F(x, u(x), Du(x)) dx$$
where we recall $F(x, z, p)= F^0(x, z) + F^1(x, p)$.

Note that $\frac{u+ v}{2}\in \bar{S}[\varphi,\Omega_0]$. From (\ref{F01}) and (\ref{F11}), we find
$$F^0(x, u(x)) + F^0(x, v(x))\geq 2 F^0\left(x, \frac{u(x) + v(x)}{2}\right) - \frac{C_b}{4}|u(x)- v(x)|^2~\forall x\in\Omega_0$$
and
$$F^1(x, Du(x)) + F^1(x, Dv(x))\geq 2 F^1\left(x, \frac{Du(x) + Dv(x)}{2}\right) + \frac{\underbar{C}}{4}|Du(x)- Dv(x)|^2~\forall x\in\Omega_0.$$
Adding these inequalities and integrating over $\Omega_0$, we find that
\begin{eqnarray}
\label{uvcnx}
J(u) + J(v) &\geq& 2 J(\frac{u+v}{2}) + \frac{\underbar{C}}{4}\int_{\Omega_0}|Du(x)- Dv(x)|^2 dx- \frac{C_b}{4}\int_{\Omega_0}|u(x)- v(x)|^2 dx\nonumber\\
&\geq& 2 J(\frac{u+v}{2}) + \frac{1}{4}\int_{\Omega_0}|Du(x)- Dv(x)|^2 dx.
\end{eqnarray}
In the last inequality of (\ref{uvcnx}), we used (\ref{tineq}) while recalling (\ref{bC}) and $u=v$ on $\p\Omega_0$. By the minimality of $u$ and $v$, we deduce from (\ref{uvcnx}) that $u\equiv v$.  Therefore,  (\ref{p1}) has  unique minimizer in $\bar{S}[\varphi,\Omega_0]$ as asserted.
\end{proof}
{\bf Acknowledgements.} The author would like to thank the referee for carefully reading the paper and providing 
critical suggestions that help improve the exposition of the paper.


\begin{thebibliography}{xx}
\bibitem{Ab} Abreu, M. K\"ahler geometry of toric varieties and extremal metrics.  {\it Int. J.
Math.} {\bf 9} (1998), no. 6, 641--651.
 \bibitem{CLR1} Carlier, G.; Lachand-Robert, T.  Regularity of solutions for some variational problems subject to a convexity constraint. {\it Comm. Pure Appl. Math.} {\bf 54} (2001), no. 5, 583--594.
\bibitem{CR} Carlier, G.; Radice, T. Approximation of variational problems with a convexity constraint by PDEs of Abreu type.
{\it Calc. Var. Partial Differential Equations.} {\bf 58} (2019), no. 5, Art. 170, 13 pp. 
\bibitem{EG} Evans, L. C.; Gariepy, R. F. {\em Measure theory and fine properties of functions.} Studies in Advanced Mathematics. CRC Press, Boca Raton, FL, 1992. 
\bibitem{Le_JDE} Le, N. Q. $W^{4, p}$ solution to the second boundary value problem of the prescribed affine mean curvature and Abreu's equations. {\it J. Differential Equations} {\bf 260} (2016), no. 5, 4285--4300. 
\bibitem{Le18} Le, N. Q. Singular Abreu equations and minimizers of convex functionals with a convexity constraint, arXiv:1811.02355v3, to appear in {\it  Comm. Pure Appl. Math.};  doi: 10.1002/cpa.21883
\bibitem{RC} Rochet, J.-C.; Chon\'e, P. Ironing, sweeping and multidimensional screening.  {\it Econometrica} {\bf 66} (1998), no. 4,  783--826.
 \bibitem{T} Tobasco, I. Curvature-driven wrinkling of thin elastic shells, arXiv:1906.02153 [math.AP].
\bibitem{TW00} Trudinger, N.S., Wang, X.J. 
The Bernstein problem for affine maximal hypersurfaces. {\it Invent. Math.} {\bf 140} (2000), no. 2, 399--422.

\end{thebibliography}
\end{document}